\documentclass[11pt,a4paper]{amsart}
\usepackage{fullpage}
\usepackage{epsfig, amsmath,xspace}
\usepackage{amssymb,amscd}
\usepackage[all,cmtip]{xy}
\usepackage{hyperref}

\newtheorem{thm}{Theorem}[section]

\newtheorem{prop}[thm]{Proposition}
\newtheorem{cor}[thm]{Corollary}

\theoremstyle{definition}
\newtheorem{rem}[thm]{Remark}
\newtheorem{defn}[thm]{Definition}

\def\ie{\emph{i.e.}}
\def\eg{\emph{e.g.}}
\newdir{ >}{{}*!/-7pt/@{>}}
\newcommand{\sab}{{\sf smod}_R}

\newcommand{\chain}{{\sf Ch}_{\geq 0}}

\newcommand{\aw}{{\sf aw}}
\newcommand{\sh}{{\sf sh}}
\newcommand{\Sh}{{\mathrm{Sh}}}
\newcommand{\hot}{\hat{\otimes}}
\newcommand{\hottau}{\hat{\tau}}
\newcommand{\ttau}{\tilde{\tau}}

\newcommand{\uR}{\underline{R}}
\newcommand{\ttensor}{\tilde{\otimes}}
\newcommand{\tcom}{\mathrm{Com}_{\ttensor}}
\def\C{\mathcal{C}}
\def\D{\mathcal{D}}

\def\B{\mathcal{B}}

\def\BP{\mathit{BP}}

\def\Z{\mathbb{Z}}
\def\F{\mathbb{F}}
\def\lra{\longrightarrow}
\def\id{\mathrm{id}}
\def\leq{\leqslant}
\def\geq{\geqslant}

\def\ra{\rightarrow}

\def\bp1{\BP\langle 1\rangle}
\begin{document}
\title[Divided power structures and chain complexes]{Divided power
  structures and chain complexes}
\author{Birgit Richter}
\address{Department Mathematik der Universit\"at Hamburg,
Bundesstra{\ss}e 55, 20146 Hamburg, Germany}
\email{richter@math.uni-hamburg.de}
\thanks{
I thank Benoit Fresse, Haynes Miller and Teimuraz Pirashvili for help
with some of the references. \today}
\begin{abstract}
We interpret divided power structures on the homotopy groups of
simplicial commutative rings as having a counterpart in
divided  power structures on chain complexes coming from a
non-standard symmetric monoidal structure. 
\end{abstract}
\maketitle

\section{Introduction}
Every commutative simplicial algebra has a divided power structure on
its homotopy groups. The Dold-Kan correspondence compares simplicial
modules to non-negatively graded chain complexes. It is an equivalence
of categories, but its multiplicative properties do not interact well
with commutativity: commutative simplicial algebras are sent to
homotopy commutative differential graded algebras, but in general not
to differential graded algebras that are commutative on the nose. 
The aim of this note is to gain a better understanding
when suitable multiplicative structures on a chain complex actually do
give rise to divided power structures. To this end, we use the equivalence of
categories between simplicial modules and non-negatively graded chain
complexes and transfer the tensor product of simplicial modules to a
symmetric monoidal category structure on the category of chain
complexes. 

We start with a brief overview on divided power algebras in section
\ref{sec:pdalgebras}. We prove a general transfer result for symmetric
monoidal category structures in section
\ref{sec:equivalence}. That such a transfer of monoidal structures is
possible is a folklore result and constructions like ours are used in
other contexts, see for instance \cite[p.263]{Sch} and \cite{Q}. 
We consider the case of chain complexes in section \ref{sec:divpowers}
where  we use this monoidal structure to gain our main results: 
in Corollary \ref{cor:divhomology} we give a criterion when a chain
complex  has a divided power structure on its
homology groups and in Theorem \ref{thm:commttensor} we describe when
we can actually gain a divided power structure on the differential graded 
commutative algebra that interacts nicely with the differential. 

As an example we give an alternative description
of the well-known (\cite{C,BK}) divided power structure on Hochschild
homology: instead of working with the bar construction in the
differential graded setting, we consider the simplicial bar
construction of a commutative algebra. This is naturally a simplicial
commutative algebra and hence the reduced differential graded bar
construction  inherits a divided power chain algebra structure from
its simplicial relative. 

\section{Divided power algebras} \label{sec:pdalgebras}
Let $R$ be a commutative ring with unit and let  $A_*$ be an 
$\mathbb{N}_0$-graded commutative algebra with $A_0=R$. We denote the
positive part of $A_*$, $\bigoplus_{i>0} A_i$, by $A_{*>0}$.
\begin{defn}
\label{def:gradeddivpowers}
A \emph{system of divided powers in $A_*$} consists of a collection of
functions $\gamma_n$ for $n \geq 0$ that are defined on $A_i$ for $i
> 0$ such that the following conditions
are satisfied.   
\begin{enumerate}
\item
$\gamma_0(a) = 1$  and $\gamma_1(a) = a$ for all $a \in A_{*>0}$. 
\item 
The degree of $\gamma_i(a)$ is $i$ times the degree of $a$. 
\item
$\gamma_i(a) = 0$ if the degree of $a$ is odd and $i>1$. 
\item
$\gamma_i(\lambda a) = \lambda^i \gamma_i(a)$ for all $a \in A_{*>0}$
and $\lambda \in R$. 
\item
For all $a \in A_{*>0}$
$$ \gamma_i(a) \gamma_j(a) = \dbinom{i+j}{i} \gamma_{i+j}(a).$$
\item
For all $a, b \in  A_{*>0}$
$$ \gamma_i(a+b) = \sum_{k+\ell = i} \gamma_k(a)\gamma_\ell(b).$$
\item
For all $a, b \in  A_{*>0}$ 
$$ \gamma_i(ab) = i! \gamma_i(a) \gamma_i(b) = a^i \gamma_i(b) =
\gamma_i(a) b^i.$$
\item
For all $a\in  A_{*>0}$
$$ \gamma_i(\gamma_j(a)) = \frac{(ij)!}{i!(j!)^i} \gamma_{ij}.$$
\end{enumerate}
\end{defn}
If we want to specify a fixed system of divided powers $(\gamma_i)_{i
  \geq 0}$ on $A_*$, we use the notation $(A_*, \gamma)$.  
For basics about systems of divided powers see \cite[Expos\'e 7,
8]{C}, \cite[section 7]{GL}, \cite[Appendix 2]{E}, \cite[Chapitre
I]{B} and \cite{Ro}. 

\subsubsection*{Some  properties}

Condition (e) implies that for all $i > 1$ the $i$-fold power of an
element $a \in A_{*>0}$ is related to its $i$-th divided power via 
\begin{equation}\label{eq:powersdivp}
a^i = i! \gamma_i(a). 
\end{equation}
Therefore, if the underlying $R$-modules $A_i$ are torsion-free, then there
is at most one system of divided powers on $A_*$, and if $R$ is a
field of  characteristic zero, then the assignment
$\gamma_i(a) =  a^i/i!$ defines a unique system of divided powers on every
$A_*$. 

If squares of odd degree elements are zero and if we are in a
torsion-free context, then condition (c) is of course taken care of by
condition (e). Some authors (\eg  \,\cite{C}) demand that the underlying
graded commutative algebra is \emph{strict}, \ie, that $a^2 = 0$
whenever $a$ has odd degree. 

Not every $\mathbb{N}_0$-graded commutative algebra possesses a system
of divided powers. Consider for instance the polynomial rings $\Z[x]$
over $\Z$ and $\F_2[x]$ over $\F_2$, where $x$ is a generator in
degree two. In $\F_2[x]$,
$x^2$ is not zero, but if there were a system of divided powers, the
equation $x^2 = 2\gamma_2(x)$ would force $x^2$ to vanish. Over the
integers, the existence of $\gamma_2(x)$ would imply that $x^2$ were
divisible by two and that's not the case. 

Note, that the following useful product formula
\begin{equation} \label{eq:prod} 
\frac{(ij)!}{i!(j!)^i} = 
\prod_{r=2}^{i} \dbinom{rj-1}{j-1}
\end{equation}
holds. 

We saw that over the rationals, divided powers can be expressed in
terms of the underlying multiplication of the $\mathbb{N}_0$-graded
commutative algebra. If the ground ring $R$ is a field of characteristic $p$
for some prime number $p \geq 2$, then for any system of divided powers on
$A_*$ the relation 
$$ a^p = p! \gamma_p(a)$$
forces the $p$-th powers of elements in $A_*$ to be trivial. There are
more relations implied by divided power structures, for instance any
iteration of the form $\gamma_i(\gamma_p(a))$ is equal to 
$\gamma_{ip} \frac{(ip)!}{i!(p!)^i}$, but using relation
\eqref{eq:prod} it is easy to see that the coefficient of
$\gamma_{ip}$ is congruent to one and hence
$$ \gamma_i(\gamma_p(a)) = \gamma_{ip}(a).$$
For a more thorough treatment of divided powers in prime
characteristic see \cite[Expos\'e 7, \S\S 7,8]{C}, \cite{A}, and  \cite{G}.

\subsubsection*{Divided power structure with respect to an ideal} 
The occurence of divided power structures is not limited to the graded
setting. In an $\mathbb{N}_0$-graded commutative $R$-algebra $A_*$ the positive
part is an ideal. In the context of ungraded commutative rings, divided power
structures can be defined relative to an ideal. The following
definition is taken from \cite[Chapitre I Definition 1.1.]{B}.
\begin{defn}
Let $A$ be a commutative ring and $I$ and ideal in $A$. A
\emph{divided power structure on $I$} consists of a family of maps 
$$ \gamma_i \colon I \ra A, i \geq 0$$ which satisfy the following
conditions 
\begin{enumerate}
\item 
For all $a \in I$, $\gamma_0(a) = 1$ and $\gamma_1(a) = a$. The image
of the $\gamma_i$ for $i\geq 2$ is contained in $I$. 
\item
For all elements $a \in A$ and $b \in I$, $\gamma_i(ab) =
a^i\gamma_i(b)$. 
\item
Conditions (e), (f) and (h) of Definition \ref{def:gradeddivpowers} apply in an
adapted sense. 
\end{enumerate}
\end{defn}
An important example of divided power structures on ungraded rings is
the case of discrete valuation rings of mixed characteristic. If $p$
is the characteristic of the residue field, $\pi$ is a uniformizer 
and $p = u \pi^e$ with $u$ a unit, then for the existence of a divided
power structure on the discrete valuation ring is is necessary and
sufficient that the ramification index $e$ is less than or equal to
$p-1$ (see \cite[Chapitre I, Proposition 1.2.2]{B}). 

\subsubsection*{Morphisms and free objects}

 Morphisms are straightforward to define:
\begin{defn}
Let $(A_*, \gamma)$ and $(B_*, \gamma')$ be two $\mathbb{N}_0$-graded
commutative algebras with systems of divided powers. A
morphism of  $\mathbb{N}_0$-graded commutative algebras $f\colon A_*
\ra B_*$ is a \emph{morphism of divided power structures}, if 
$$ f(\gamma_i(a)) = \gamma'_i(f(a)), \quad \text{ for all } \quad i
\geq 2.$$
\end{defn}
An analogous definition works in the ungraded case. 

We will describe the free divided power algebra
generated by  a $\mathbb{N}_0$-graded module $M_*$ whose components
$M_i$ are free $R$-modules. 

\begin{defn}
Consider the free $R$-module generated by an element $x$ of degree
$m$. 
\begin{itemize}
\item 
If $m$ is odd, then \emph{the free divided power algebra on $x$} over
$R$ is the
exterior algebra over $R$ generated by $x$, $\Lambda_R(x)$. In this
case the $\gamma_i$ are trivial for $i \geq 2$. 
\item
If $m$ is even, \emph{the free divided power algebra on $x$ over $R$} is
$$ R[X_1,X_2,\ldots]/I.$$
Here the $X_n$ are polynomial generators in degree $nm$ and $I$ is the
ideal generated by 
$$ X_i X_j - \dbinom{i+j}{i} X_{i+j}.$$
\end{itemize}
As the tensor product over $R$ is the coproduct in the category of
$\mathbb{N}_0$-graded commutative $R$-algebras, we get a notion of a free
divided power algebra on a finitely generated module $M_*$ whose $M_i$
are free as $R$-modules by taking care of Condition (g) of Definition
\ref{def:gradeddivpowers}. If $M_*$ is not finitely generated we take
the colimit of the free divided power algebras on finitely many
generators. Compare \cite[Proposition 1.7.6]{GL}. 

If $M_*$ is an  $\mathbb{N}_0$-graded module that is freely generated
by  elements $x_1,\ldots,
x_n$, then it is common to denote the free divided power algebra over
$R$ on these generators by $\Gamma_R(x_1, \ldots, x_n)$. 
\end{defn}

Occurences of free divided power algebras are ample. For instance,
the cohomology ring of the loop space on a sphere,
$\Omega(\mathbb{S}^n)$, for $n \geq 2$ is a free divided power
algebra. Using the Serre spectral sequence for the path-loop
fibration $\Omega \mathbb{S}^n \ra P\mathbb{S}^n \ra \mathbb{S}^n$ one
gets 
$$ H^*(\Omega(\mathbb{S}^n); \Z) \cong \left\{
\begin{array}{ll}
\Gamma_{\Z}(a) & |a| = n-1, n \text{ odd }\\
\Gamma_{\Z}(a,b)  \cong \Lambda_{\Z}(a) \otimes \Gamma_{\Z}(b) & |a| =
n-1, |b| = 2n-2, n \text{ even }
\end{array}
\right. $$

Often, free divided power algebras arise as duals of symmetric
algebras. For a $\mathbb{N}_0$-graded $R$-module $M_*$,  its $R$-dual,
$M^*$, is the $\mathbb{N}_0$-graded $R$-module with $M^i = \mathrm{Hom}_R(M_i,
R)$. The symmetric algebra generated by a $\mathbb{N}_0$-graded
$R$-module $M_*$ is 
$$ S(M_*) = T_R(M_*)/C$$
where $C$ is the ideal generated by the graded commutativity relation 
$$ ab - (-1)^{|a||b|}ba.$$
Here, $|a|$ denotes the degree of $a$ and $T_R(M_*) = \bigoplus_{i\geq
  0} M_*^{\otimes i}$ is the tensor algebra generated by $M_*$. 

The diagonal map $M_* \ra M_* \oplus M_*$, $a \mapsto (a,a)$ induces a
cocommutative coalgebra structure 
$$\Delta \colon S(M_*) \lra S(M_* \oplus M_*) \cong S(M_*) \otimes
S(M_*). $$

On the graded dual of $S(M_*)$, $S(M_*)^*$, this comultiplication
yields a graded commutative multiplication
$$ m = \Delta^* \colon S(M_*)^* \otimes S(M_*)^* \lra S(M_*)^*. $$
If $M_*$ is a $\mathbb{N}_0$-graded $R$-module whose components are
free over  $R$, then $S(M_*)^*$ has a system of divided powers (see
for  instance \cite[A2.6]{E}). 

Let $\Sigma_n$ denote the symmetric group on $n$-letters. If $N_*$ is an
is an $\mathbb{N}_0$-graded module with $\Sigma_n$-action, then we
denote by  $N_*^{\Sigma_n}$ the invariants in $N_*$ with respect to the
$\Sigma_n$-action. 
If $M$ is a free $R$-module, then one can describe the free divided
power algebra on $M$ as
\begin{equation}
  \label{eq:invariants}
  \bigoplus_{n \geq 0} (M_*^{\otimes n})^{\Sigma_n}.
\end{equation}
This is a classical result and is for instance proved in
\cite[Expos\'e 8, Proposition 4]{C}. See Roby \cite[Remarque
p.~103]{Ro} for an example where the two notions differ if one
considers a module that is not free. Divided power structures can in
fact be described via \eqref{eq:invariants}: a graded module with free
components, $M_*$, with $M_0=0$ has a divided power structure if there is a map 
\begin{equation}
  \label{eq:monad}
 \bigoplus_{n \geq 1} (M_*^{\otimes n})^{\Sigma_n} \ra M_* 
\end{equation}
that satisfies the axioms of a monad action (see \cite{F}). The monad
structure that is applied in the desciption via \eqref{eq:monad} uses
the  invertibility of the norm map
on reduced symmetric sequences of the form $M_*^{\otimes n}$
\cite[1.1.16 and 1.1.18]{F}. The invertibility of the norm map in this
case was discovered earlier by  Stover \cite[9.10]{St}.

\subsubsection{Divided power structures in the simplicial context}

On the homotopy groups of  simplicial commutative rings there are
divided power operations and it is this instance of divided power
structures that we will investigate in this paper. 

In the context of the action of the Steenrod algebra on cohomology
groups of spaces, the top operation in the $p$-th power map. On the
homotopy groups of simplicial commutative $\F_2$-algebras, there are
analogous operations $\delta_i$ of degree $i\geq 2$ such that the highest
operation is the divided square. These operations were investigated by
Cartan \cite[Expos\'e no 8]{C} and were intensely studied by many
people (\cite[section 8]{Bo}, \cite{D}, \cite [chapter 2]{G}, 
\cite{T}). In \cite[8.8 onwards]{Bo} a family of operations for odd primes is
discussed as well.

\subsection*{Notation}
With $\Delta$ we denote the category whose
objects are the sets $[n] = \{0,\ldots,n\}$ with their natural
ordering and morphisms from in $\Delta$ are monotone maps. A
simplicial object in a category $\mathcal{C}$ is a functor from the
opposite category of $\Delta$, $\Delta^{op}$, to $\mathcal{C}$. We
denote the category of simplicial objects in $\mathcal{C}$ by
$s\mathcal{C}$. 

In the category of simplicial sets, the representable functors 
$$ \Delta(n)\colon \Delta^{op} \rightarrow \mathrm{Sets}$$
are the ones that send $[m] \in \Delta$ to $\Delta([m],[n])$. 
If $\delta_i\colon [n] \ra [n+1]$ denotes the map that is the
inclusion that misses $i$ and is strictly monotone everywhere else
and if $\sigma_i\colon [n] \ra [n-1]$ is the surjection that sends
$i$ and $i+1$ to $i$ and is strictly monotone elsewhere, then we
denote their opposite maps by $d_i = (\delta_i)^{op}$ and $s_i =
(\sigma_i)^{op}$.

If $S$ is a set and $R$ is a ring, then we denote the free $R$-module 
generated by $S$ by $R[S]$. The tensor product of two $R$-modules $N$
and $M$, $N \otimes_R M$, will be abbreviated by $N \otimes M$. 

\section{The Dold-Kan correspondence} \label{sec:doldkan}
The Dold-Kan correspondence  \cite[Theorem 1.9]{D1} 
compares the category of simplicial
objects in an abelian category $\mathcal{A}$ to the non-negatively
graded chain complexes over $\mathcal{A}$ via a specific equivalence
of categories. In the following we fix an arbitrary commutative ring
with unit $R$ and we  will focus on the correspondence
between simplicial $R$-modules , $\sab$, and non-negatively graded
chain complexes of $R$-modules, $\chain^R$.

The equivalence is given by the
\emph{normalization functor}, $N\colon \sab \ra \chain^R$, and we denote
its inverse by $\Gamma$
$$ \xymatrix{{N\colon \sab} \ar@<0.5ex>[r] &
{\chain^R \, \colon\!\Gamma}\ar@<0.5ex>[l].  } $$
In particular the functor $N$ is a left
adjoint to $\Gamma$. The value of
$N$ on a  simplicial $R$-module $X_\bullet$ in chain degree $n$ is 
$$ N_n(X_\bullet) = \bigcap_{i=1}^{n} \mathrm{ker}(d_i: X_n \lra X_{n-1}) $$
where the $d_i$ are the simplicial structure maps. The differential
$d: N_n(X_\bullet) \ra N_{n-1}(X_\bullet)$ is given by the remaining face map
$d_0$. 


Recall that for a chain complex $C_*$, 
$$\Gamma_n(C_*) = \bigoplus_{p=0}^n \bigoplus_{\varrho\colon [n]
  \twoheadrightarrow [p]} C_p^\varrho $$
where $\varrho$ is an order preserving surjection and $C_p^\varrho =
C_p$. The normalized chain complex has an alternative description via a
quotient construction where one reduces modulo degenerate
elements \cite[lemma 8.3.7]{W}. 
There is a canonical identification
$\varphi_{C_*} \colon N\Gamma C_* \cong C_*$:  if $\varrho$ is 
not the identity map, then the simplicial structure of $\Gamma(C_*)$
identifies elements in $C_p^\varrho$ as being degenerate. For a
simplicial $R$-module $A_\bullet$ the isomorphism $\psi_{A_\bullet}
\colon \Gamma N(A_\bullet) \cong A_\bullet$ is induced by the map that
sends $N(A_\bullet)_p^\varrho \subset A_p$ via $\varrho$ to $A_n$.

The tensor product of chain complexes $(C_*,d)$ and $(C'_*,d')$ is defined as
usual via 
$$ (C_* \otimes C'_*)_n = \bigoplus_{p+q=n} C_p \otimes C'_q$$
with differential $D(c \otimes c') = (dc)\otimes c' + (-1)^p c \otimes
d'c'$ for $c \in C_p$, $c' \in C'_q$. 
Let $(R,0)$ denote the chain complex that has $R$ as
degree zero part and that is trivial in all other degrees. There is a twist
isomorphism $\tau_{C_*,C'_*}\colon C_* \otimes C'_* \ra C'_* \otimes
C_*$ that is induced by $\tau_{C_*,C'_*}(c\otimes c') = (-1)^{pq} c'
\otimes c$ for $c$ and $c'$ as above. The structure $(\chain^R, \otimes,
(R,0), \tau)$ turns $\chain^R$ into a symmetric
monoidal category. 

For two arbitrary simplicial $R$-modules $A_\bullet$ and
$B_\bullet$, let $A_\bullet \hot B_\bullet$ denote the degree-wise
tensor product  of $A_\bullet$ and
$B_\bullet$, \ie, $(A_\bullet \hot B_\bullet)_n = A_n \otimes_k B_n$.
Here, the simplicial structure
maps are applied in each component; in particular, the differential on
$N_*(A_\bullet \hot B_\bullet)$ in degree $n$ is $d_0 \otimes d_0$. The
constant simplicial object $\underline{R}$ which consists of
$R$ in every
degree is the unit with respect to $\hot$ and the twist 
$$ \hottau_{A_\bullet, B_\bullet} \colon A_\bullet \hot B_\bullet \lra 
B_\bullet \hot A_\bullet, \qquad  \hottau_{A_\bullet,
  B_\bullet}(a\otimes b) = b \otimes a$$ gives 
$(\sab, \hot, \uR, \hottau)$ the structure of a symmetric monoidal 
category. Note that $N(\uR) \cong (R,0)$. 

There are natural maps, the \emph{shuffle maps}, 
$$\sh\colon N(A_\bullet) \otimes
N(B_\bullet) \lra N(A_\bullet \hot B_\bullet)$$
(see \cite[VIII.8]{ML}) that turn the normalization into a lax
symmetric monoidal functor, \ie, the shuffle maps are associative in a
suitable sense and the diagram 
$$\xymatrix{
{N(A_\bullet) \otimes N(B_\bullet)} \ar[r]^{\sh} \ar[d]_{\tau} 
& {N(A_\bullet \hot B_\bullet)} \ar[d]^{N(\hottau)}\\
{N(B_\bullet) \otimes N(A_\bullet)} \ar[r]^{\sh} & {N(B_\bullet \hot A_\bullet)}
}$$
commutes for all $A_\bullet, B_\bullet \in  \sab$. 
However, the inverse of $N$, $\Gamma$, is
\emph{not} lax symmetric monoidal. In order to compare $\Gamma(C_*)
\hot \Gamma(C'_*)$ and $\Gamma(C_* \otimes C'_*)$ one uses the 
Alexander-Whitney map
$$\aw\colon N(A_\bullet \hot B_\bullet) \lra N(A_\bullet) \otimes
N(B_\bullet)$$
and this natural map is not symmetric, \ie, the diagram 
$$\xymatrix{
{N(A_\bullet \hot B_\bullet)}\ar[r]^(0.45){\aw} \ar[d]_{N(\ttau)} &  
{N(A_\bullet) \otimes N(B_\bullet)} \ar[d]^{\tau}\\
{N(B_\bullet \hot A_\bullet)}\ar[r]^(0.45){\aw} & {N(B_\bullet) \otimes
N(A_\bullet)}
}
$$
does \emph{not} commute. 

Schwede and Shipley proved, that the Dold-Kan correspondence passes to
a  Quillen equivalence between the category of associative simplicial
rings and the category  of differential graded associative
algebras that are concentrated in non-negative degrees. They consider
the  normalization functor
and construct an adjoint on the level of monoids which then gives rise
to a monoidal Quillen equivalence \cite{SchSh}.

If one starts with a differential graded commutative algebra, then
$\Gamma$ sends this algebra to a simplicial $E_\infty$-algebra
\cite[Theorem 4.1]{R}. In general, the Dold-Kan correspondence gives rise to a
Quillen adjunction between simplicial homotopy $\mathcal{O}$-algebras
and differential graded homotopy $\mathcal{O}$-algebras for operads
$\mathcal{O}$ in $R$-modules  \cite[Theorem 5.5.5]{R2}. 

\section{The divided power structure on homotopy groups of  
commutative simplicial algebras} 

In the following we view $\Sigma_{ni}$ as the group of
bijections of the set $\{0, \ldots, ni-1\}$. For a permutation
$\sigma$ we use $\varepsilon(\sigma)$ for its signum. We consider the set of
shuffle permutations, $$\mathrm{Sh}(\underbrace{n,\ldots,n}_{i}) \subset
\Sigma_{ni}.$$
This set consists of  permutations $\sigma \in \Sigma_{ni}$ such that 
$$ \sigma(0) < \ldots < \sigma(n-1), \ldots, \sigma((i-1)n)< \ldots <
\sigma(ni-1). $$
Let $\underline{j}$ denote the block of numbers $(j-1)n < \ldots <
jn-1$ for $1 \leq j \leq i$ and let $[ni-1] \backslash
\underline{j}$ denote the complement of $\underline{j}$ with its
inherited ordering from the one of $[ni-1]$. We use the abbreviation  
$s_{\sigma([ni] \backslash \underline{j})}$ 
for the composition of the degeneracy
maps $s_{\sigma(k)}$ where $k \in [ni-1] \backslash \underline{j}$ and
the order of the composition uses small indices first. For example,
let  $\sigma \in \mathrm{Sh}(2,2,2)$ be the permutation
$\sigma = (0,2)(1,4)(3,5)$
$$\xymatrix@R=0.5cm@C=0.5cm{
0 \ar@{-}[drr] & 1 \ar@{-}[drrr] & 2 \ar@{-}[dll] & 3 \ar@{-}[drr] & 4
\ar@{-}[dlll] & 5 \ar@{-}[dll]\\
0 & 1 & 2 & 3 & 4 & 5.
}$$
In this case, $s_{\sigma([5]
  \backslash \underline{2})} = s_4 \circ s_3 \circ s_2 \circ s_1.$

Let $A_\bullet$ be a commutative simplicial $R$-algebra, \ie, a commutative
monoid in $(\sab, \hot, \uR, \hottau)$. The homotopy groups of
$A_\bullet$, $\pi_*(A_\bullet)$, are the homology groups of the normalization of
$A_\bullet$, $H_*(N(A_\bullet))$. Starting with a cycle 
$a \in N_n(A_\bullet)$ we can map 
$a$ it to its $i$-fold tensor power
$$  N_n(A_\bullet) \ni a \mapsto a^{\otimes i} \in
N_n(A_\bullet)^{\otimes i}.$$
The $i$-fold iterated shuffle map sends $a^{\otimes i}$ to 
$$ \sum_{\sigma \in \mathrm{Sh}(n,\ldots,n)} 
\varepsilon(\sigma) s_{\sigma([ni-1] \backslash \underline{1})}(a)
\otimes \ldots \otimes s_{\sigma([ni-1] \backslash \underline{n})}(a)$$
so that the outcome is an element in 
$$A_{ni} \otimes \ldots \otimes A_{ni} = (A_\bullet \hot \ldots \hot
A_\bullet)_{ni}.$$
As none of the degeneracy maps arises $n$ times, we consider the image 
as an element of  $N_{ni}(A_\bullet \hot \ldots \hot
A_\bullet)$. If we compose the $i$-fold diagonal map with the
$i$-fold iterated shuffle map followed by the commutative
multiplication in $A_\bullet$, we can view the composite as a map 
$$ P_i\colon  N_n(A_\bullet) \lra N_{ni}(A_\bullet). $$ 
A tedious calculation shows that $P_i$ is actually a chain map. 
On the level of homology, this composite sends a homology class to its 
$i$-fold power.

The group $\Sigma_i$ acts on the set of shuffles
$\mathrm{Sh}(\underbrace{n,\ldots,n}_{i})$ by permuting the
$i$ blocks of size $n$. If $\xi \in \Sigma_i$ we denote the corresponding block
permutation by $\xi^b$. For a 
$\sigma \in \mathrm{Sh}(\underbrace{n,\ldots,n}_{i})$ and $\xi \in
\Sigma_i$, $\sigma \circ \xi^b$ is again an element of 
$\mathrm{Sh}(\underbrace{n,\ldots,n}_{i})$. 
As $A_\bullet$ is commutative, we have that 
the multiplication applied to a summand 
$s_{\sigma([ni-1] \backslash \underline{1})}(a)
\otimes \ldots \otimes s_{\sigma([ni-1] \backslash \underline{n})}(a)$
gives the same output as the multiplication applied to the summand
corresponding to $\sigma \circ \xi^b$. 

If the characteristic of the ground ring is not two, then 
$i$-fold powers for $i \geq 2$ are trivial unless $n$ is
even. The signum of the permutation $\sigma \circ \xi^b$ is the signum of
$\xi^b$ multiplied by $\varepsilon(\sigma)$. For each crossing in 
$\xi$ the block permutation $\xi^b$ has $n^2$ crossings, so 
$\xi^b$ is in the alternating group in this case. If the
characteristic of $R$ is $2$, then signs do not matter.

\begin{defn} \label{def:divsimpl}
For a simplicial commutative $R$-algebra $A_\bullet$, the \emph{$i$-th divided
power}  of $[a] \in \pi_n(A_\bullet) = H_n(N(A_\bullet))$ is defined
as the class of 
$$ \mu \circ \sum_{\sigma \in \mathrm{Sh}(n,\ldots,n)/\Sigma_i} 
\varepsilon(\sigma) s_{\sigma([ni] \backslash \underline{1})}(a)
\otimes \ldots \otimes s_{\sigma([ni] \backslash \underline{n})}(a)$$
where we choose a system of representing elements $\sigma \in
\mathrm{Sh}(n,\ldots,n)/\Sigma_i$. We denote the $i$-th divided power
of $a \in \pi_n(A_\bullet)$ by $\gamma_i(a)$. 
\end{defn}

As we know that all elements $\sigma'$ in the same coset as $\sigma$
give rise to the same value under the map $P_i$, we obtain that 
$$ a^i = i! \gamma_i(a).$$ 
With the conventions $\gamma_0(a) = 1$ and
$\gamma_1(a) = a$ we obtain the following (see for instance \cite[\S 2.2]{F}
for a proof). 
\begin{prop}
The system of divided powers in the homotopy groups of a commutative simplicial
$R$-algebra $A_\bullet$, $(\pi_*(A_\bullet), \gamma)$ satisfies the
properties from Definition
\ref{def:gradeddivpowers}. 
\end{prop}

\section{A large symmetric monoidal product on chain
  complexes}
We will use the following product later in order to investigate
divided power structures on chain complexes. 
\begin{defn}
We define the \emph{large tensor product} of two chain complexes $C_*$
and $C'_*$ to be 
$$ C_* \ttensor C'_* := N(\Gamma(C_*) \hot \Gamma(C'_*)).$$
\end{defn}
Note that the large tensor product deserves its name: the degenerate
elements in $\Gamma(C_*) \hot
\Gamma(C'_*)$ are only the ones that are images of the maps $s_i \hot
s_i$, and in general $N(\Gamma(C_*) \hot \Gamma(C'_*))$ is much larger
than the ordinary tensor product $C_* \otimes C'_* \cong  
N(\Gamma(C_*)) \otimes N(\Gamma(C'_*))$. 

As a concrete example, consider the normalized chain complex on
the standard simplex $\Z[\Delta(1)]$. Let us denote a monotone map $f
\colon [n] \ra [1]$ by an $(n+1)$-tupel corresponding to its image, so
that for instance 
$$f\colon [3] \ra [1], f(0) = 0, f(1) = f(2) = f(3) = 1$$
is represented by $(0,1,1,1)$. 

As $\Z[\Delta(1)]$ has
non-degenerate simplices only  in degrees zero and one corresponding to
the monotone maps $(0)$ and $(1) \in \Delta([1],[0])$ and $(0,1)$ in
$\Delta([1], [1])$, its normalization  
$C_* = N(\Z[\Delta(1)])$ is the chain complex
$$ \Z \oplus \Z \leftarrow \Z \leftarrow 0 \leftarrow  \ldots$$
and the boundary map sends the generator $(0,1)$ to $(0) -
(1)$. Therefore, $C_* \otimes C_*$ is a chain complex, that is
concentrated in degrees zero, one and two with chain groups of rank
four, four and one respectively. Note that 
$$ N(\Gamma(N(\Z[\Delta(1)])) \hot \Gamma(N(\Z[\Delta(1)]))) \cong
N(\Z[\Delta(1)] \hot \Z[\Delta(1)]).$$
Thus for instance in degree one, $C_* \ttensor
C_*$ is of rank seven.    
\section{Equivalences of categories and transfer of monoidal
  structures} \label{sec:equivalence}
If $F\colon \mathcal{C} \ra \mathcal{D}$ and $G\colon \mathcal{D} \ra 
\mathcal{C}$ is a pair of functors that constitute an equivalence of 
categories and if $(\mathcal{D}, \hot, 1, \hottau)$ is symmetric monoidal,
then we can transfer the symmetric monoidal structure on $\D$ to one
on $\C$ in the following way. 

\begin{itemize}
\item
As for chain complexes, one defines a 
product $\ttensor$ via $ C_1 \ttensor C_2 = G(FC_1 \hot FC_2)$ for
objects $C_1, C_2$ of $\C$. 
\item
As an equivalence 
$$ \ttau_{C_1,C_2} \colon C_1 \ttensor C_2 = G(FC_1 \hot FC_2) \ra  
G(FC_2 \hot FC_1) = C_2 \ttensor C_1$$
we take $G(\hottau_{FC_1,FC_2})$.
\item
The unit for the symmetric monoidal structure is $G(1)$.
\end{itemize}
For later reference, we spell out some of the structural
isomorphisms. Recall that any equivalence of categories gives rise to
an adjoint equivalence \cite[IV.4]{ML}; in particular the unit and
counit of the adjunction are isomorphisms. We want to denote the
natural isomorphism from $GFC$ to $C$ for $C$ an object of $\C$ by
$\varphi_C$ and the one from $FGD$ to $D$ by $\psi_D$ for all objects
$D$ in $\D$. Then the identities 
\begin{equation}
  \label{eq:isos}
F(\varphi_C) = \psi_{FC} \text{ and } G(\psi_D) = \varphi_{GD}  
\end{equation}
hold for all $C$ and $D$.

For the left unit we have to identify  $C$ with 
$G(1) \ttensor C$ and to this end we use the morphism 
$$\xymatrix@1{{\tilde{\ell} \colon C} \ar[r]^{\varphi^{-1}} & {G(F(C))}
  \ar[r]^(0.45){G(\hat{\ell})} & {G(1 \hot F(C))}
  \ar[rr]^(0.35){G(\psi^{-1} \hot \id)}  & & {G(F(G(1)) \hot
    F(C)) = G(1) \ttensor C}}$$
where $\hat{\ell}$ is the left unit isomorphism for $\hot$. The right
unit is defined similarly. 

The associativity isomorphism $\tilde{\alpha}$ is given in terms of
the one for $\hot$, $\hat{\alpha}$ as  
$$ \tilde{\alpha} := G(\id \hot \psi)^{-1} \circ G(\hat{\alpha}) \circ
G(\psi \hot \id):$$

$$\xymatrix{
{G(FG(F(C) \hot F(C)) \hot
 F(C))} \ar[dd]_{G(\psi \hot \id)}
\ar[rr]^{\tilde{\alpha}} & &{G(F(C)\hot FG(F(C) 
\hot F(C)))} \ar[dd]^{G(\id \hot \psi)} \\
 & & \\
{G((F(C) \hot F(C)) \hot F(C))} 
\ar[rr]^{G(\hat{\alpha})} & &{G(F(C) \hot (F(C) \hot
  F(C)))}  
}$$

Then it is a tedious, but straightforward task to show the following
result. A proof in the non-symmetric setting can be found in
\cite[Theorem 3]{Q}. 
\begin{prop}
The category $(\C, \ttensor, G(1), \ttau)$ is a symmetric monoidal
category.  
\end{prop}

If $\C$ already has a symmetric monoidal structure, then we can
compare the old one to the new one as follows. 
\begin{prop} \label{prop:lsm}
If $(\C, \otimes, G(1), \tau)$ is a symmetric monoidal structure and
if $G\colon (\D, \hot, 1, \hottau) \ra (\C, \otimes, G(1), \tau)$ is
lax symmetric monoidal, then the identity functor 
$$ \id \colon (\C, \ttensor, G(1), \ttau) \lra 
(\C, \otimes, G(1), \tau)$$ 
is lax symmetric monoidal. 
\end{prop}
\begin{proof}
We have to construct maps $\lambda_{C_1, C_2} \colon C_1 \otimes C_2 
\rightarrow C_1 \ttensor C_2$ that are natural in $C_1$ and $C_2$ and
that render the diagrams 
$$ \xymatrix{
{C_1 \otimes C_2} \ar[d]_{\tau} \ar[r]^{\lambda_{C_1,C_2}} & {C_1 \ttensor C_2}
\ar[d]^{\ttau_{C_1,C_2}}\\
{C_2 \otimes C_1} \ar[r]^{\lambda_{C_2,C_1}} & {C_2 \ttensor C_1}
}$$ 
commutative for all $C_1, C_2 \in \C$. Let $\Upsilon$ be the
transformation that turns $G$ into a lax symmtric monoidal functor. 
We define 
$$ \xymatrix@1{{\lambda_{C_1,C_2}\colon C_1 \otimes C_2} 
\ar[rr]^{\varphi^{-1}_{C_1} \otimes \varphi^{-1}_{C_2}} & &
{GF(C_1) \otimes GF(C_2)} \ar[rr]^(0.45){\Upsilon_{F(C_1), F(C_2)}} & 
& {G(F(C_1) \hot  F(C_2)) = C_1 \ttensor C_2}}.$$
\end{proof}
\begin{cor}
Every commutative monoid in $(\C, \ttensor, G(1), \ttau)$ is a
commutative monoid in $(\C, \otimes, G(1), \tau)$.
\end{cor}
The functor $F$ compares commutative monoids in the categories $\C$
and $\D$  as follows. 
\begin{thm} \label{thm:gammac}
An object $F(C)$ is a commutative monoid in $(\D,\hot,1,\hottau)$ if
and  only if $C$ is a commutative monoid in
$(\C, \ttensor, G(1), \ttau)$. Moreover, the assignment
$C \mapsto F(C)$ is a functor from the category of 
commutative monoids in $(\C, \ttensor, G(1), \ttau)$ to
the category of commutative monoids in $(\D,\hot,1,\hottau)$. 
\end{thm}
\begin{proof}
If we assume that $C$ is a commutative monoid in $(\C, \ttensor, G(1),
\ttau)$, then $C$ has an associative 
multiplication
$$\tilde{\mu} \colon C \ttensor C = G(F(C)\hot F(C))
\lra C$$
that satisfies $\tilde{\mu} \circ \ttau = \tilde{\mu}$ and there is a
unit map $j\colon G(1) \ra C$. We consider the
composition $\varphi^{-1} \circ \tilde{\mu}\colon G(F(C)\hot
F(C)) \ra GF(C)$. As $G$ is an equivalence of
categories, it is a full functor, \ie, the morphism $\varphi^{-1} \circ
\tilde{\mu}$ is of the form $G(\hat{\mu})$ for some morphism 
$\hat{\mu}\colon F(C) \hot
F(C) \ra F(C)$ in $\D$. We will show that $\hat{\mu}$ turns
$F(C)$ into a commutative monoid. We define the unit
map as $i = F(j)\circ \psi^{-1} \colon 1 \ra F(C)$.

As $\ttau = G(\hottau)$, the commutativity of $\hat{\mu}$ follows from
the one  of
$\tilde{\mu}$ and the fact that the functor $G$ is faithful. 

In order to check the unit property of $i$ we have to show that the
following diagram commutes: 
\begin{equation} \label{diag:leftunit}
\xymatrix{
{F(C)} \ar@{=}[d]\ar[r]^{\hat{\ell}} & {1 \hot F(C)}
\ar[rr]^{\psi^{-1} \hot \id} & & {F(G(1)) \hot F(C)} 
\ar[d]^{F(j) \hot \id} \\
{F(C)} & & & {F(C) \hot F(C).} \ar[lll]^{\hat{\mu}}
}
\end{equation}
As $j$ is a unit for the multiplication $\tilde{\mu}$ we know that 
$$ \tilde{\mu} \circ (j \ttensor \id) \circ G(\psi^{-1} \hot \id)
\circ G(\hat{\ell}) \circ \varphi_C^{-1} = \id_{C}.$$
Applying the faithful functor $F$ to
this identity and using the 
definition of $\hat{\mu}$, we get that 
$$   FG(\hat{\mu}) \circ F(j \ttensor \id) \circ F 
G(\psi^{-1} \hot \id) \circ
FG(\hat{\ell})  = \id_{FGF(C)}.$$
By the very definition, $F(j \ttensor \id)$ is 
$ F(G(F(j) \hot \id))$ and thus via the faithfulness of $FG$ we can
conclude that diagram \eqref{diag:leftunit} commutes. The analogous
statement for the right unit can be shown similarly and hence $i$ is a unit. 

For the associativity of the multiplication $\hat{\mu}$ we
have to show that the inner pentagon in the following diagram
commutes. 
$$\xymatrix@C=0pt{
{G(F G(F(C) \hot F(C)) \hot
  F(C))}\ar@/_6pc/[ddd]_{\tilde{\mu} \ttensor \id}
\ar[dd]^{G(\psi \hot \id)}
\ar[rr]^{\tilde{\alpha}} & &{G(F(C)\hot F
  G(F(C) \hot F(C)))} \ar[dd]_{G(\id \hot \psi)} 
\ar@/^6pc/[ddd]^{\id \ttensor \tilde{\mu}}\\
 & & \\
{G((F(C) \hot F(C)) \hot F(C))} 
\ar[d]^{G(\hat{\mu} \hot \id)}
\ar[rr]^{G(\hat{\alpha})} & &{G(F(C) \hot (F(C) \hot
  F(C)))}  \ar[d]_{G(\id \hot \hat{\mu})}\\
{G(F(C) \hot F(C))} \ar[dr]_{G(\hat{\mu})} \ar@/_2pc/[ddr]_{\tilde{\mu}}& 
&{G(F(C) \hot F(C))}\ar[dl]^{G(\hat{\mu})} \ar@/^2pc/[ddl]^{\tilde{\mu}}\\
&{GF(C)} \ar[d]^{\varphi}&\\
&{C}&
}$$
The outer diagram commutes because $\tilde{\mu}$ is associative and
the upper square commutes because $\tilde{\alpha}$ is given in terms
of $\hat{\alpha}$ in this way. The only thing that remains to be
proven is that the outer wings commute. We prove the claim for the
left wing. Using the definitions  of the maps involved, we have to
show that 
$$ \xymatrix{
{G(F G(F(C) \hot F(C))\hot
  F(C))}\ar[dr]_{\tilde{\mu} \ttensor \id =
  G(F(\tilde{\mu}) \hot \id) \phantom{blabla}} \ar[rr]^{G(\psi_{F(C)\hot F(C)} \hot
  \id)} & {} & {G((F(C) \hot F(C))\hot F(C))}
\ar[dl]^{G(\hat{\mu}\hot \id)}
 \\
{} & {G(F(C)\hot F(C))} & {}
}$$
commutes. We know that 
$$F G(\hat{\mu}) =
F(\varphi^{-1}_{C}) \circ F(\tilde{\mu})$$
and therefore 
$$ \psi_{F(C)} \circ F G(\hat{\mu}) =
F(\tilde{\mu}).$$
The naturality of $\psi$ implies that 
$$ \psi_{F(C)} \circ F G(\hat{\mu}) = \hat{\mu} \circ 
(\psi_{F(C) \hot F(C)})$$
and the claim follows.  

If $F(C)$ is a commutative monoid with respect to $\hot$ with multiplication 
$\hat{\mu}$ and unit $i \colon 1 \ra F(C)$, then  we claim that 
$\tilde{\mu} := \varphi_{C} \circ G(\mu)$ and 
$j := \varphi_{C} \circ G(i)$ give $C$ the structure of 
a commutative monoid  with respect to $\ttensor$. 

The fact that $\tilde{\mu}$ is commutative follows directly because
$\ttau = G(\hottau)$. The proof of the unit axiom and of associativity
use the same diagrams as above with the arguments reversed. 

It remains to show that a morphism $f \colon C_1 \ra C_2$ of
commutative monoids with respect to $\ttensor$ gives rise to a
morphism $F(f)\colon F(C_1) \ra F(C_2)$ of commutative monoids with
respect to $\hot$. It suffices to show that the outer diagram in 
$$\xymatrix{
{G(F(C_1) \hot F(C_1))} \ar@/_3pc/[dd]_{G(\hat{\mu})} \ar[rr]^{G(F(f) \hot
  F(f))} \ar[d]_{\tilde{\mu}}  & {} & {G(F(C_2) \hot F(C_2))}
\ar[d]^{\tilde{\mu}}  \ar@/^3pc/[dd]^{G(\hat{\mu})}\\
{C_1} \ar[rr]^{f} & {} & {C_2}  \\
{GF(C_1)} \ar[u]^{\varphi_{C_1}} \ar[rr]^{GF(f)}& {} &
{GF(C_2)} \ar[u]_{\varphi_{C_2}}
}$$
commutes. The upper square commutes by assumption, the lower square
commutes because $\varphi$ is natural and the wings commute by the
very defintion of $\hat{\mu}$ in terms of $\tilde{\mu}$.

\end{proof}

\section{Divided power structures and commutative monoids}
\label{sec:divpowers}

If we apply the above results to the Dold-Kan 
correspondence with $\Gamma=F$, $N=G$, $\mathcal{C}= \chain^R$
and $\mathcal{D}= \sab$ and the large tensor product of chain complexes, then
we obtain the following statements that we collect in one theorem. 

\begin{thm} \label{thm:meta}
\begin{enumerate}
\item[]
\item
The category of chain complexes with the large tensor product is a
symmetric monoidal category with $N(\uR)$ being the unit of
the monoidal structure and  
$$ \ttau = N(\hottau_{\Gamma(C_*), \Gamma(C'_*)}) \colon 
N(\Gamma(C_*) \hot \Gamma(C'_*)) \lra N(\Gamma(C'_*) \hot
\Gamma(C_*))$$
as the twist. 
\item
The identity functor 
$$ \id \colon (\chain^R, \ttensor, N(\uR), \ttau) \lra 
(\chain^R, \otimes, (R,0), \tau)$$ 
is lax symmetric monoidal.
\item
Every commutative monoid in $(\chain^R, \ttensor, N(\uR),
\ttau)$ is a differential graded commutative $R$-algebra.
\item
A simplicial $R$-module $\Gamma(C_*)$ is a simplicial commutative
$R$-algebra if and only if the chain complex $C_*$ is a commutative monoid in
$(\chain^R, \ttensor, N(\uR), \ttau)$. The assignment
$C_* \mapsto \Gamma(C_*)$ is a functor from the category of 
commutative monoids in $(\chain^R, \ttensor, N(\uR), \ttau)$ to
the category of simplicial commutative $R$-algebras. 
\end{enumerate}
\end{thm}

\begin{cor} \label{cor:divhomology}
Every commutative monoid $C_*$ in $(\chain^R, \ttensor, N(\uR),
\ttau)$ has a divided power structure on its homology. 
\end{cor}

The converse of statement (c) of Theorem \ref{thm:meta} is
\emph{not}  true:  not
every differential graded commutative algebra possesses a divided
power structure on its homology, so these algebras cannot be
commutative monoids in $(\chain^R, \ttensor, N(\uR),
\ttau)$. For instance the polynomial ring
$\Z[x]$ over the integers with $x$ of degree two and trivial
differential provides an example. 

\begin{defn} \label{def:tcom}
We denote by $\tcom$ the category whose objects are commutative
monoids in $(\chain^R, \ttensor, N(\uR), \ttau)$ and whose morphisms are
multiplicative chain maps, \ie, if $C_*,D_*$ are objects of $\tcom$
with unit maps $j_{C_*}\colon N(\uR) \ra C_*$ and $j_{D_*}\colon N(\uR) \ra D_*$,
then a morphism is a chain map $f\colon C_* \ra D_*$ with $f
\circ j_{C_*} = j_{D_*}$ and such that the
diagram
$$\xymatrix{
{C_* \ttensor C_*}  \ar[d]_{\tilde{\mu}_{C_*}} \ar[r]^{f \ttensor f} & {D_*
    \ttensor D_*}  \ar[d]^{\tilde{\mu}_{D_*}} \\
{C_*} \ar[r]^{f} & {D_*}
}$$
commutes. 
\end{defn}
If we start with a $C_* \in \tcom$, then in particular, $C_*$ is a differential
graded algebra, \ie, the differential $d$ on the underlying chain complex
$C_*$ is compatible with the product structure: it satisfies the
Leibniz rule
$$  d(ab) = d(a)b + (-1)^{|a|}a d(b), \quad \text{for all} \quad a, b
\in C_*.$$
If there are divided power structures on the underlying graded
commutative algebra $C_*$, then we want these to be compatible with
the differential. 
\begin{defn} \label{def:divchain}
1) A commutative differential algebra $C_*$ with divided power operations
is called a \emph{divided power
chain algebra}, if the differential $d$ of $C_*$ satisfies
\begin{enumerate}
\item 
$$ d(\gamma_i(c)) = d(c)\cdot \gamma_{i-1}(c) \quad \text{for all} \quad c 
\in C_*$$
\item
If $c$ is a boundary, then $\gamma_i(c)$ is a boundary for all $i \geq
1$. 
\end{enumerate}

2) A morphism of commutative differential algebras $f\colon C_* \ra
D_*$ is a \emph{morphism of divided power chain algebras}, if $f$
satisfies $f(\gamma_i(c)) = \gamma_i(f(c))$ for all $c \in C_*$, $i
\geq 0$.  
\end{defn}
The first condition in 1) ensures that divided powers respect cycles and
together with the second condition this guarantees that the homology
of $C_*$ inherits a divided power structure from $C_*$. 

A reformulation of the criterium for a divided power chain algebra is
used in \cite[definition 1.3]{AH}: they demand that every element of
positive degree is in the image of a morphism of differential graded
commutative algebras with divided power structure $f \colon D_* \ra
C_*$ such that $D_*$ satisfies condition (a) of Definition
\ref{def:divchain} and has trivial homology in positive degrees. 


For a commutative monoid with respect to $\ttensor$, $C_*$, the $i$-th
power of an element $c \in N\Gamma(C_*)$ is given via the following
composition 
$$ 
\xymatrix{
{c \in } \ar@{|->}[dd] &  {N\Gamma(C_*)_n} \ar[r]^(0.4){c \mapsto 
  c^{\otimes i}} & {N\Gamma(C_*)_n^{\otimes i}  
\subset (N\Gamma(C_*)^{\otimes i})_{ni}}
\ar[d]^{\sh} \\ 
{} & {} & {N(\Gamma(C_*) \hot \ldots \hot \Gamma(C_*))_{ni} = C_* \ttensor
  \ldots \ttensor C_*} \ar[dl]_(0.6){N(\hat{\mu})} \ar[d]^{\tilde{\mu}}  \\
{c^i \in } & {N\Gamma(C_*)_{ni}} \ar[r]^{\varphi}& {C_*.}
}
$$
We define a divided power structure on $N\Gamma(C_*)$ by using a
variant of the shuffle map as in Definition \ref{def:divsimpl} sending
$c^{\otimes i}$ to 
$$ \sum_{\sigma \in \mathrm{Sh}(n,\ldots,n)/\Sigma_i} 
\varepsilon(\sigma) s_{\sigma([ni] \backslash \underline{1})}(c)
\otimes \ldots \otimes s_{\sigma([ni] \backslash \underline{n})}(c)$$ 
and applying $N(\mu)$. 
\begin{thm} \label{thm:commttensor}
The composite $N \Gamma$ is a functor from the category $\tcom$ to the
category of divided power chain algebras.  
\end{thm}
\begin{proof}
Let $C_*$ be a commutative monoid in $(\chain^R, \ttensor, N(\uR),
\ttau)$.  Let us first prove that 
$$ d(\gamma_i(c)) = \gamma_{i-1}(c) \cdot d(c)$$
for all $c \in C_*$ in positive degrees. 
If we apply the boundary $d=d_0$ to 
\begin{equation}\label{eq:sum} 
\gamma_i(c) = N(\mu)(\sum_{\sigma \in \Sh(n,\ldots,n)/\Sigma_i}
\varepsilon(\sigma)  s_{\sigma([ni-1] \backslash \underline{1})}(c)
\otimes \ldots \otimes s_{\sigma([ni-1] \backslash
\underline{n})}(c))
\end{equation}
then we can use that $d_0$ is a morphism in the simplicial category to
obtain 
$$ d_0 \circ N(\mu) = N(\mu) \circ d_0 \otimes \ldots \otimes d_0.$$
Only one of the sets $\sigma([ni-1] \backslash \underline{j})$ does
not contain zero. Therefore the simplicial identities  
$d_0 \circ s_i = s_{i-1} \circ d_0$ for
$i>0$ and $d_0 \circ s_0 = \id$ ensure, that in the sum \eqref{eq:sum}
there are $i-1$ tensor factors containing  just degeneracies applied
to $c$ and only one term containing degeneracies applied to $d_0(c)$. 

A shuffle permutation in $\Sh(\underbrace{n,\ldots,n}_{i-1})$
tensorized with the identity map followed by a shuffle in
$\Sh(n(i-1),n)$ gives a shuffle in $\Sh(\underbrace{n,\ldots,n}_{i})$
and every shuffle  in $\Sh(\underbrace{n,\ldots,n}_{i})$ is
decomposable in the above way. There are 
$$ \prod_{j=2}^n \frac{\binom{nj}{n}}{i!}$$
elements in $\Sh(\underbrace{n,\ldots,n}_{i})/\Sigma_i$ and 
$$ \prod_{j=2}^{n-1} \frac{\binom{nj}{n}}{(i-1)!} \cdot \dbinom{ni-1}{n-1}$$
elements in the product of $\Sh(\underbrace{n,\ldots,n}_{i-1})/\Sigma_{i-1}$ and
$\Sh(n(i-1),n)$. As these numbers are equal, we obtain that the two
sets are in bijection and $d_0(\gamma_i(c))$ can be expressed as the
product of $\gamma_{i-1}(c)$ and $d_0(c)$. 

\medskip

The boundary criterium for the divided power structure can be seen as
follows: if $c$ is of
the form $d_0(b)$ for some $b \in N\Gamma(C_*)_{n+1}$, then $\gamma_i(c)$
is 
$$ N(\mu)(\sum_{\sigma \in \Sh(n,\ldots,n)/\Sigma_i}
\varepsilon(\sigma)  s_{\sigma([ni-1] \backslash \underline{1})}(d_0(b))
\otimes \ldots \otimes s_{\sigma([ni-1] \backslash
\underline{n})}(d_0(b))).$$ 
The simplicial identity $s_{j-1}d_0 = d_0s_j$
for all $j>0$ allows us to move the $d_0$-terms in front by increasing the
indices of the degeneracy maps.  Therefore $\gamma_i(d_0(b))$ is equal to
an expression of the form $ N(\mu)N(d_0 \otimes \ldots \otimes d_0)(x)$ for some
suitable $x$. As $N(\mu)N(d_0 \otimes \ldots \otimes d_0)$ is equal to
$d_0 \circ N(\mu)$ we obtain the desired result.  

\medskip
For $f\colon C_* \ra D_*$ a morphism of commutative
monoids in $(\chain^R, \ttensor, N(\uR), \ttau)$ we have to show
that $N\Gamma(f)$ is a morphism of divided power chain algebras, \ie,
that it is a multiplicative chain map that preserves units and divided
powers. As $f$ is a chain map, $\Gamma(f)$ is a map of simplicial
$R$-modules and $N\Gamma(f)$ is a chain map. If $j_{C_*}$ and
$j_{D_*}$ are the units for $C_*$ and $D_*$, we have $f \circ j_{C_*}
= j_{D_*}$ and this implies 
$$ N\Gamma(f) \circ N\Gamma(j_{C_*}) \circ \varphi^{-1}_{N(\uR)} =
N\Gamma(j_{D_*}) \circ \varphi^{-1}_{N(\uR)}$$
and thus the unit condition holds.

In order to establish
that $N\Gamma(f)$ is multiplicative we have to show that the back face
in the diagram
$$ \xymatrix{
{N\Gamma(C_*) \otimes N\Gamma(C_*)} \ar[rrrr]^(0.4){\mu} \ar[drr]^{\sh} 
\ar[ddd]_{N\Gamma(f) \otimes N\Gamma(f)}& {} & {} & {} &
{N\Gamma(C_*)} \ar[ddd]_{N\Gamma(f)}\\
{} & {} & {N(\Gamma(C_*) \hot \Gamma(C_*))} \ar[rru]^{
  \varphi^{-1}
  \circ \tilde{\mu} = N(\hat{\mu}) \phantom{blab}} 
\ar[ddd]_(0.35){N(\Gamma(f) \hot \Gamma(f))} & {} & {} \\
{} & {} & {} & {} & {} \\
{N\Gamma(D_*) \otimes N\Gamma(D_*)} \ar'[rr][rrrr]^(-0.3){\mu} \ar[drr]^{\sh}& {}
& {} &  {} & {N\Gamma(D_*)} \\
{} & {} & {N(\Gamma(D_*) \hot \Gamma(D_*))} \ar[rru]^{
  \varphi^{-1}
  \circ \tilde{\mu} = N(\hat{\mu}) \phantom{blab}} & {} & {} 
}$$
commutes. The top and bottom triangle commute by definiton, the left
front square commutes because the shuffle map is natural and the right
front square commutes because we know from Theorem \ref{thm:gammac} that
$\Gamma(f)$ is multiplicative. 

The fact that $f$ preserves divided powers can be seen directly: using
naturality of $\varphi$ and the multiplicativity of $f$ with respect
to $\ttensor$, the only thing we have to verify is the compatibility
of $f$ with the variant of the shuffle map. But this map is a sum of
tensors of degeneracy maps and as $\Gamma(f)$ respects the simplicial
structure, the claim follows. 

\end{proof} 

\begin{rem}
We can transfer the model structure on
simplicial commutative $R$-algebras as in \cite[II, Theorem 4]{Quillen} to
the  category $\tcom$ of commutative monoids in $(\chain^R, \ttensor, N(\uR),
\ttau)$ by declaring that 
a  map $f\colon C_* \ra D_*$ is a weak
equivalence, fibration resp.~cofibration in $\tcom$ if and only if
$\Gamma(f)$ is a weak
equivalence, fibration resp.~cofibration in the model structure on
commutative simplicial $R$-algebras.  

\end{rem}

\section{Bar constructions and Hochschild complex} 

One well-known example of a divided power chain algebra is the
normalization of a bar construction of a commutative $R$-algebra (see for
instance \cite[Expos\'e no 7]{C} and \cite[\S 3]{BK}).

Let $A$ be a commutative $R$-algebra. The bar construction of $A$ is the
simplicial commutative $R$-algebra, $\B_\bullet(A)$, with 
$$\B_n(A) = A^{\otimes (n+2)}. $$
The simplicial structure maps are given by inserting the
multiplicative unit $1 \in R$ for
degeneracies and by multiplication for face maps. As $A$ is
commutative, we can multiply componentwise
$$ \B_n(A) \otimes \B_n(A) \ra \B_n(A), \qquad (a_0 \otimes \ldots
\otimes a_{n+1}) \otimes (a'_0 \otimes \ldots
\otimes a'_{n+1}) \mapsto a_0a'_0 \otimes \ldots
\otimes a_{n+1}a'_{n+1}.$$

From Theorems \ref{thm:gammac} and \ref{thm:commttensor} it follows
that the normalization $B_*(A) := N(\B_\bullet(A))$ is a divided power
chain algebra. 
 
The Hochschild complex of the  
commutative $R$-algebra $A$ ist  defined as 
$$ C_*(A) = A \otimes_{A \otimes A} B_*(A)$$
where the $A$-bimodule structure on $B_n(A)$ is induced by
$$ (a \otimes \tilde{a}) (a_0 \otimes \ldots
\otimes a_{n+1}) := aa_0 \otimes \ldots
\otimes a_{n+1}\tilde{a}.$$

If $A$ is flat over $R$, the homology of this complex is $\mathrm{Tor}_*^{A
  \otimes A}(A,A)$.  
As $B_*(A)$ is acyclic and surjects onto $C_*(A)$, the Hochschild
complex inherits a structure of a divided power chain algebra from
$B_*(A)$. Cartan showed \cite[Expos\'e 7]{C}, that the bar
construction of strict differential graded commutative algebras has a
divided power structure. Condition (b) of Definition
\ref{def:divchain} is in general satisfied on the iterated bar
construction $\B_*^{(n)}, n \geq 2$ \cite[p. 7, Expos\'e 7]{C}, so
that each $\B_*^{(n)}, n \geq 2$ is a divided power chain algebra.


\begin{thebibliography}{90999999}

\bibitem[A76]{A}
Michel Andr\'e, \emph{Puissances divis\'ees des alg\`ebres simpliciales en  
caract\'eristique deux et s\'eries de Poincar\'e de certains anneaux
locaux}, Manuscripta Math. \textbf{18} (1976), no. 1, 83--108.

\bibitem[AH86]{AH}
Luchezar Avramov, Stephen Halperin, 
\emph{Through the looking glass: a
dictionary between rational homotopy theory and local algebra},  
Algebra, algebraic topology and their interactions (Stockholm, 1983), 
Lecture Notes in Math. \textbf{1183}, Springer, Berlin (1986), 1--27.

\bibitem[Be74]{B}
Pierre Berthelot, 
\emph{Cohomologie cristalline des sch\'emas de caract\'eristique
  $p>0$}, Lecture Notes in Mathematics {\bf 407}, Springer-Verlag,  
Berlin-New York (1974) 604 pp.


\bibitem[Bo67]{Bo}
Aldridge K. Bousfield, 
\emph{Operations on derived functors for non-additive functors},
unpublished notes, Brandeis University (1967) 69 pp. 


\bibitem[BK94]{BK}
Siegfried Br\"uderle, Ernst Kunz, \emph{Divided powers and Hochschild
homology of complete intersections}, with an appendix by Reinhold
H\"ubl, Math. Ann. \textbf{299} (1994), no. 1, 57--76.

\bibitem[C54]{C}
Henri Cartan, 
S\'eminaire Henri Cartan, 1954-55, Alg\`ebre d'Eilenberg-MacLane et
homotopie, available at \url{http://www.numdam.org/}


\bibitem[Do58]{D1} 
Albrecht Dold, {\em Homology of symmetric products and other
functors of complexes}, Ann. of Math. (2) \textbf{68} (1958),  54--80.

\bibitem[D80]{D} 
William G. Dwyer, \emph{Homotopy operations for simplicial commutative
algebras}, Trans. Amer. Math. Soc.  \textbf{260} (1980), no. 2, 421--435.

\bibitem[EM54]{EM}
Samuel Eilenberg, Saunders Mac Lane, \emph{On the groups $H(\Pi,n)$, II,
Methods of computation},  Ann. of Math. (2)  \textbf{60}  (1954), 49--139.

\bibitem[E95]{E}
David Eisenbud, 
\emph{Commutative algebra. With a view toward algebraic
geometry}, 
Graduate Texts in Mathematics \textbf{150}, Springer-Verlag, New
York  (1995), xvi+785 pp. 


\bibitem[F00]{F}
Benoit Fresse, \emph{On the homotopy of simplicial algebras over an
operad},  Trans. Amer. Math. Soc.  \textbf{352} (2000),  no. 9, 4113--4141.

\bibitem[G90]{G}
Paul Goerss, 
\emph{On the Andr\'e-Quillen cohomology of commutative
  $F_2$-algebras}, 
Ast\'erisque  {\bf 186} (1990), 169 pp.


\bibitem[GL69]{GL}
Tor H. Gulliksen, Gerson Levin, 
\emph{Homology of local rings}, 
Queen's Paper in Pure and Applied Mathematics {\bf 20},  
Queen's University, Kingston, Ont. (1969),  x+192 pp. 


\bibitem[ML95]{ML}
Saunders Mac Lane, \emph{Homology}, 
Reprint of the 1975 edition, Classics in Mathematics,  
Springer-Verlag, Berlin, (1995) x+422 pp.



\bibitem[Q$\infty$]{Q}
Nguyen Tien Quang, 
\emph{On GR-Functors between GR-Categories}, preprint 
arXiv:0708.1348. 

\bibitem[Qui67]{Quillen}
Daniel G. Quillen, 
\emph{Homotopical algebra}, 
Lecture Notes in Mathematics \textbf{43} Springer-Verlag, Berlin-New
York (1967)  iv+156 pp.

\bibitem[R03]{R}
Birgit Richter, \emph{Symmetry properties of the Dold-Kan
  correspondence}, Math. Proc. Cambridge Philos. Soc.  \textbf{134}
(2003), no. 1, 95--102. 

\bibitem[R06]{R2}
Birgit Richter, \emph{Homotopy algebras and the inverse of the
  normalization  functor}, Journal of Pure and Applied Algebra
\textbf{206} (3) (2006), 277--321.


\bibitem[Ro68]{Ro}
Norbert Roby, 
\emph{Construction de certaines alg\'ebres \`a  puissances
  divis\'ees}, Bull. Soc. Math. France  \textbf{96} (1968), 97--113.

\bibitem[Sch01]{Sch}
Peter Schauenburg, 
\emph{Turning monoidal categories into strict ones}  
New York J. Math.  \textbf{7}  (2001), 257--265. 

\bibitem[SchSh03]{SchSh}
Stefan Schwede, Brooke Shipley, 
\emph{Equivalences of monoidal model categories},
Algebr. Geom. Topol. \textbf{3} (2003), 287--334.  

\bibitem[St93]{St}
Christopher R. Stover, 
\emph{The equivalence of certain categories of twisted Lie and Hopf
  algebras over a commutative ring}, 
J. Pure Appl. Algebra  \textbf{86}  (1993),  no. 3, 289--326.



\bibitem[T99]{T}
James M. Turner, James M. Relations in the homotopy of simplicial abelian Hopf
algebras.  J. Pure Appl. Algebra  134  (1999),  no. 2, 163--206. 

\bibitem[W94]{W}
Charles A. Weibel, 
\emph{An introduction to homological algebra}, 
Cambridge Studies in Advanced Mathematics \textbf{38}, 
Cambridge University Press, Cambridge (1994) xiv+450 pp.

\end{thebibliography}
\end{document}